\documentclass[amsthm]{elsarticle}
\usepackage{lipsum}
\makeatletter
\def\ps@pprintTitle{%
 \let\@oddhead\@empty
 \let\@evenhead\@empty
 \def\@oddfoot{}%
 \let\@evenfoot\@oddfoot}
\makeatother

\usepackage{multirow} % for table multirow
\usepackage{enumitem} % for parsep, itemsep...
\usepackage{xspace} % for \xspace

\usepackage{graphicx}
\usepackage{amsmath,amssymb,amsthm}
\usepackage{bm}
\usepackage{pdfsync}
\usepackage{subfigure}
\usepackage[dvipsnames]{xcolor}
\usepackage[english]{babel}
\usepackage{pictexwd}

\theoremstyle{plain}
\newtheorem{theorem}{Theorem}[section]
\newtheorem{proposition}[theorem]{Proposition}
\newtheorem{lemma}[theorem]{Lemma}
\newtheorem{corollary}[theorem]{Corollary}

\theoremstyle{definition}
\newtheorem{example}[theorem]{Example}
\newtheorem{definition}[theorem]{Definition}
\newtheorem{remark}[theorem]{Remark}

\def\tprod{\mathop{\textstyle\prod}\limits}
\def\LT{\mathop{\rm LT}\nolimits}
\def\m{{\mathfrak{m}}}

\def\q{{\mathfrak{q}}}

\def\Stair{{\rm Stair}}

\def\soc{{\rm soc}}
\def\Ann{{\rm Ann}}

\def\maxgrid{{\rm mgrid}}
\def\Eqq#1{\mathop{=\joinrel=}\limits^{{#1}}}
\def\card{{\rm card}}

\def\cocoa{\mbox{\rm
  C\kern-.13em o\kern-.07 em C\kern-.13em o\kern-.15em A}}

\def\GFN{\mathop{\rm GFNum}\nolimits}
\def\Supp{\mathop{\rm Supp}\nolimits}

\def\gb{Gr\"obner basis }

\usepackage{color}
\definecolor{OliveGreen}{cmyk}{0.64,0,0.95,0.40}
\definecolor{Purple}{cmyk}{0.45,0.86,0,0}

\begin{document}

\begin{frontmatter}

\title{Small Gr\"obner Fans of Ideals of Points}

\author{Elena Dimitrova\fnref{myfootnote}}
\address{\scriptsize Department of Mathematical Sciences, Clemson University, Clemson, SC 29634, USA }
\ead{edimit@clemson.edu}

\author{Qijun He\fnref{myfootnote}}
\address{Biocomplexity Institute and Initiative, Charlottesville, VA 22911, USA}
\ead{qh4nj@virginia.edu}

\author{Lorenzo Robbiano}
\address{\scriptsize Dip. di Matematica,
\ Universit\`a degli Studi di Genova, \ Via
Dodecaneso 35,\
I-16146\ Genova, Italy}
\ead{robbiano@dima.unige.it}

\author{Brandilyn Stigler\fnref{myfootnote}}
\address{\scriptsize Department of Mathematics, Southern Methodist University, Dallas, TX 75275, USA }
\ead{bstigler@smu.edu}

\fntext[myfootnote]{Partially supported by NSF grant DMS-1419023.}

\bigskip
\begin{abstract}
In the context of modeling biological systems, it is of interest to generate
ideals of points with a unique reduced Gr\"obner basis, and
the first main goal of this paper is to identify classes of ideals
in polynomial rings which share this property. Moreover, we provide  
methodologies for constructing such ideals.
We~then relax the condition of uniqueness. The second and most relevant
topic discussed here is to consider and identify pairs of ideals with the
same number of reduced Gr\"obner bases, that is, with the same
cardinality of their associated Gr\"obner~fan.
\end{abstract}

\date{\today}


\end{frontmatter}

%\tableofcontents

%%%%%%%%%%%%%%%%%%%%%%%%%%%%%%%%
%%%%%%%%%%%%%%%%%%%%%%%%%%%%%%%%
%%%%%%%%%%%%%%%%%%%%%%%%%%%%%%%%
\section{\large Introduction}
%Gr\"obner bases have been used in a variety of applications.  In one application,
%they are used to  select minimal models for biological networks.
Gr\"obner bases have enjoyed a diverse set of applications since their inception in 1965 (for example, see \cite{lin2004, maniatis2007, torrente2009, tsai2016}).  In 2004, Gr\"obner bases were applied to the problem of model selection in systems biology \cite{laubenbacher2004computational}.  Specifically, they were introduced as a tool to select minimal models from a set of polynomial dynamical systems (PDS) that fit discretized experimental data: for a given set of data points over a finite field, the ideal of points forms a coset representing the space of PDSs that fit the data and a minimal model is selected from the space by computing a reduced Gr\"obner basis of the ideal and taking the normal forms of the model equations. While this provides an algorithmic solution to model selection, each choice of monomial order results in a different minimal PDS, with each one yielding different hypotheses about the underlying biological network.  The following example illustrates this claim.

\smallskip
%\begin{example}\label{lac}
Lactose metabolism in \textit{E.coli} is controlled by the \textit{lac}
operon, a genetic system made up of simultaneously transcribed genes.
It is said that the \textit{lac} operon~($x$) is ON (lactose is metabolized)
when the activating protein CAP~($y$) is present
and when the inhibiting protein \textit{lacI}($z$) is absent.
This behavior can be described by the Boolean function $f=y\wedge \lnot z$;
as a polynomial over the finite field
$\mathbb F_2$, we can write $f=y(z+1)=yz+y$.
If we consider the inputs $\mathbb X=\{(1,0,0),(0,1,0),(1,0,1)\}$
representing Boolean states for the \textit{lac} operon, CAP, and \textit{lacI}
respectively, then the ideal of polynomials vanishing  on $\mathbb X$ has two Gr\"obner bases,
namely
$\{x^2+x, z^2+z, y+x+1, xz+z\}$ and
$\{y^2+y, z^2+z, x+y+1, yz\}$.
The normal forms of $f$ are $x+1$ and $y$ respectively.
Note that the function $f$ is selected as a model using the first \gb
while a different model is selected using the second Gr\"obner basis.
%\end{example}

\smallskip
Computing all possible minimal PDSs requires computing the Gr\"obner fan of the ideal which is computationally expensive, even in the finite field case. The authors in \cite{dimitrova2014data} posed the question of finding data sets whose corresponding ideals have a small number, possibly a unique reduced Gr\"obner basis, or whose Gr\"obner fans consist of a single cone. Their motivation was a desire to minimize the number of associated models, each with a different set of predictions.    

Similar problems arise in the branch of statistics called combinatorial design of experiments
(see~\cite{R} and~\cite{KR2},Tutorial 92,  for an introduction to this topic).
In the context of a field $K$, functions which fit data in $\mathbb{X}\subseteq K^n$
lie in the coordinate ring $K[\mathbb{X}]:=K[x_1,\ldots ,x_n]/\mathcal I(\mathbb{X})$.
Then the coset $f+\mathcal I(\mathbb{X})$ describes the set of models
which fit the input data in~$\mathbb{X}$ and one model is chosen by computing
the normal form of $f\in K[x_1,\ldots ,x_n]$ with respect to a
Gr\"obner basis of the ideal of points $\mathcal I(\mathbb{X})$.
Changing term orderings results in potentially different
normal forms, \textit{i.e.} different models.

The first main goal of this paper is to identify classes of ideals
in polynomial rings which have a unique reduced Gr\"obner basis. 
In Section~\ref{Background} we introduce fundamental tools such as
G-basic sets, GFan numbers, and linear shifts
(see Definitions~\ref{basic-G-basic}, \ref{GBN(I)}, and~\ref{linearshift}).
Then it is shown in Theorem~\ref{linshifts}
that ideals related by a linear shift share the same number of G-basic sets, equivalently
the same GFan number. Finally the classical notion of an ideal of points is recalled
together with the notion of a grid of points.

Section~\ref{OneReducedGbasis} starts with Theorem~\ref{GFan=Basic} which provides a 
characterization of ideals whose GFan number is 1. Such ideals turn out to have also a unique
basic set as shown in Corollary~\ref{oneGB-oneB}. Then the important notion of a distraction
is recalled. It is shown that distractions and their linear shifts provide a large class of ideals
with GFan number equal to 1 (see Theorem~\ref{6.2.12} and Corollary~\ref{distrpoints}).
The last subsection of this section focuses on natural distractions and associated staircases.
Their strong connection is highlighted in Proposition~\ref{naturaldistraction}.

Section~\ref{Complementary Ideals} contains the most relevant results of the paper.
It is well-known that every zero-dimensional ideal in ${P = K[x_1, \dots, x_n]}$ 
contains $n$ univariate polynomials, one for each indeterminate. 
Accordingly, we consider an ideal $J$ in $P$ generated by $n$ univariate polynomials, 
one for each indeterminate, and  Definition~\ref{mainAssump2} describes 
how two ideals $I_1$ and $I_2$
which contain $J$ can be considered to be complementary with respect to $J$.
The main Theorem~\ref{sameGFan} shows that complementary ideals
have the same Gr\"obner fan, hence the same GFan number, and then Corollary~\ref{IandJ2}  provides good 
classes of complementary ideals. 
The paper is concluded in Section~\ref{Application} where some applications
of the theory developed before and some hints to future research are illustrated.

Basic definitions and results are taken from~\cite{KR1}, \cite{KR2}, and \cite{KR3}, with
examples computed in \cocoa-5\ \cite{CoCoA} to allow the interested reader to check the computations directly.

%%%%%%%%%%%%%%%%%%%%%%%%%%%%%%%%%%%%
%%%%%%%%%%%%%%%%%%%%%%%%%%%%%%%%%%%%
\medskip
\section{\large Background}
\label{Background}

Let $K$ be a field, $P = K[x_1, \dots, x_n]$ a polynomial ring,
and $I$ an ideal in~$P$. 
We recall that $\mathbb T^n$ is the monoid of power products in the
indeterminates $x_1, \dots, x_n$ and that a non-empty subset $\mathcal{O}$ of $\mathbb T^n$
is called an \textbf{order ideal}  if it is closed under division
(see~\cite{KR2}, Definition 6.4.3).
If  $\sigma$ is a term ordering,
the set $\mathbb T^n{\setminus}\LT_\sigma(I)$
is denoted by~$\mathcal{O}_\sigma(I)$.
It is well-known that $\mathcal{O}_\sigma(I)$ is an order ideal and
the  residue classes of its elements form a $K$-basis of $P/I$
(see for instance~\cite{KR1}, Corollary 2.4.11).
It is also well-known that, given $I$, there are order ideals which are not of
type $\mathcal{O}_\sigma(I)$; nevertheless the residue classes of
their elements form a $K$-basis of $P/I$.
The following example  taken from~\cite{KR2} (see Example 6.4.2)
is a case in point.

\begin{example}\label{6.4.2}
Consider the ideal $I \!= \langle x^2+xy+y^2,\, x^3,\, x^2y,\, xy^2,\,y^3\rangle$
in~$\mathbb Q[x, y]$. This ideal is symmetric with respect to
switching~$x$ and~$y$. Since the leading term of~$x^2+xy+y^2$
is either~$x^2$ or~$y^2$, the ideal~$I$ has two possible
leading term ideals, namely the ideals $J_1 = \langle x^2,\, xy^2,\, y^3\rangle$
and $J_2 = \langle x^3,\, x^2y,\, y^2\rangle$. Neither is symmetric.
Thus they do not give rise to symmetric vector space bases
of~$\mathbb Q[x,y]/I$. However, the set of terms
$\mathcal{O}= \{1,\,x,\,y,\,x^2,\,y^2\}$ is symmetric and
represents a vector space basis of~$\mathbb Q[x, y]/I$.
\index{Gr\"obner basis!breaks symmetry}
\end{example}

These considerations motivate the following definition.

\begin{definition}\label{basic-G-basic}
An order ideal $\mathcal{O}$ such that the classes of its elements form
a $K$-basis of $P/I$ is called a \textbf{basic set} for $I$.
If there exists a term ordering $\sigma$ such
that $\mathcal{O}=\mathcal{O}_\sigma(I)$, it is called a
\textbf{G-basic set} for $I$. 
If we want to specify that a G-basic set is obtained using $\sigma$,
we call it a \textbf{$\sigma$-basic set}. 
\end{definition}

%\begin{remark}\label{evalmatrix}
Let $I$ be a zero-dimensional ideal with $\dim_K(P/I)=s<\infty$,
let~$\sigma$ be a term ordering,
and let $\mathcal{O}$ be an order ideal with $s$ elements. 
The normal forms of the elements in $\mathcal{O}$ with respect to $\sigma$ are linear
combinations of the elements of $\mathcal{O}_\sigma(I)$, 
and hence can be represented by an $s\times s$ matrix, say~$M$. 
It is then clear that $\mathcal{O}$ is a basic set for $I$ if and only 
if~$M$ is invertible.

%\end{remark}

Some relations between basic sets and $\sigma$-basic sets are
described in~\cite{BOT}, Section 2.  
For zero-dimensional ideals, basic sets are the main building blocks
of the theory of border bases
(see~\cite{KR2}, Section 6.4 for the introduction to that theory)
which is outside the scope of the present paper.

\medskip

As mentioned above,  the authors in \cite{dimitrova2014data} and others raised the question
of  properties of $\mathbb{X}$ that
guarantee~$\mathcal{I}(\mathbb{X})$ has a unique reduced Gr\"obner basis, hence a unique 
G-basic set: such data sets have uniquely identifiable models. To count the number of 
G-basic sets of an ideal, we use the notion of
the \emph{Gr\"obner fan} which was introduced in~\cite{MR}. It is a subdivision of the closed non-negative orthant ~$\mathbb R^n_+$ made with a finite number
of polyhedral cones, such that
the cones are in  one-to-one correspondence to the
G-basic sets for~$I$.

\begin{definition}\label{GBN(I)}
Let $I$ be an ideal in $P$ and let ${\rm GFan}(I)$ be the Gr\"obner fan of~$I$.
The number of G-basic sets for $I$,
equivalently  the number of leading term ideals of $I$,
is called the \textbf{GFan number} of~$I$, and  is denoted by $\GFN(I)$,
since it coincides with the number of polyhedral cones in  ${\rm GFan}(I)$.
\end{definition}

We point out that the definition does not count the number of different 
reduced Gr\"obner bases, as shown with the help of the following easy examples.

\begin{example}\label{easyGFan}
Let $I = \langle f\rangle \subset K[x,y]$, where $f = x+y$ and $K$ is any field.
One can argue that for every term ordering $\{f\}$ is the reduced Gr\"obner basis.
However, for every term ordering $\sigma$ with $x>y$ we have $\LT(I) =\langle x\rangle$
and the corresponding G-basic set is $\{ y^n \mid n \in \mathbb N\}$.
For every term ordering $\sigma$ with $y>x$ we have $\LT(I) =\langle y\rangle$
and the corresponding G-basic set is $\{ x^n \mid n \in \mathbb N\}$.
Consequently we have   $\GFN(I) =2$. 
\end{example}

\begin{example}\label{countingGB}
Let $I = \langle x+y+z\rangle \subset K[x,y,z]$, where $K$ is any field.
In this case we have $\GFN(I) = 3$ since the only possible leading term ideals  of $I$
are $\langle x \rangle$, $\langle  y \rangle$, $\langle z \rangle$.
\end{example}

Since a topic of this paper  is to identify ideals which
have the same GFan number, we note that some affine 
transformations do not affect leading terms.  
This observation motivates the following definition.  

\begin{definition}\label{linearshift}
An affine transformation $\Phi: P\to P$
defined by $x_i \mapsto a_ix_i+b_i$ where 
$a_i\in K\setminus \{0\}$, $b_i\in K$ for $i =1, \dots, n$ is called a \textbf{linear shift} of~$P$.
\end{definition}

\begin{proposition}\label{linshifts}
Let $\Phi$ be a linear shift of $P$ and $I$ an ideal in $P$.
\begin{enumerate}
\item[(a)] The ideals $I$ and $\Phi(I)$ have the same G-basic sets.

\item[(b)] We have $\GFN(I) = \GFN(\Phi(I))$.
\end{enumerate}
\end{proposition}

\begin{proof}
To prove (a), let $\sigma$ be a term ordering on $\mathbb T^n$ and
let $f$ be a non-zero polynomial in~$I$. It is clear that  $\LT_\sigma(f) = \LT_\sigma(\Phi(f))$
which implies the inclusion  $\LT_\sigma(I) \subseteq \LT_\sigma(\Phi(I))$. But $\Phi$ is an isomorphism
and its inverse is also a linear shift, hence we get the other inclusion.
Consequently we have $\LT_\sigma(I) = \LT_\sigma(\Phi(I))$ for every term ordering~$\sigma$
which implies that $I$ and $\Phi(I)$ have the same G-basic sets

Claim (b) follows from (a), thereby completing the proof.
\end{proof}

\begin{example}
Let us return to the ideal $I$ in Example \ref{6.4.2}.  Consider the linear shift  $\Phi = (x + 1, y -2)$. Then 
$\Phi(I)= \langle (x+1)^2+(x+1)(y-2)+(y-2)^2,\, (x+1)^3,\, (x+1)^2(y-2),\, (x+1)(y-2)^2,\, (y-2)^3\rangle$.
Note that $\Phi(I)$ also has  two 
leading term ideals, namely the same minimally generated ideals $J_1 = \langle x^2,\, xy^2,\, y^3\rangle$ and $J_2 = \langle x^3,\, x^2y,\, y^2\rangle$ as $I$ above. Indeed, $\GFN(I)=\GFN(\Phi(I))=2$.

\end{example}

%%%%%%%%%%%%%%%%%%

In areas such as design of experiments, ideals are constructed from data as was 
described in the introduction.  As such, we call a tuple $(c_1, \dots, c_n) \in K^n$  a \textbf{point}, 
corresponding to the linear maximal ideal $\m=\langle x_1-c_1, \dots, x_n-c_n\rangle \in P$.
Furthermore the \textbf{vanishing ideal} $\mathcal{I}(\mathbb Y)$ of a finite 
set $\mathbb Y$ of $s$ points is a zero-dimensional radical ideal in~$P$ of type
$\mathcal{I}(\mathbb Y) = \m_1\cap\cdots\cap \m_s$, and
which we also call an \textbf{ideal of points}.  For an introduction to ideals of points, 
see~\cite{KR2}, Section 6.3; for methods to efficiently compute them and other
zero-dimensional ideals, see~\cite{ABKR} and \cite{AKR}.

%%%%%%%%%%%%%%%%%%%%%%%%%%%%%%%%%%%%
%%%%%%%%%%%%%%%%%%%%%%%%%%%%%%%%%%%%
\medskip
\section{\large Ideals with One Reduced Gr\"obner Basis}
\label{OneReducedGbasis}

In this section, we look for conditions which guarantee that an ideal $I$ has $\GFN(I) =1$.  
We assume that $K$ is any a field and $P = K[x_1, \dots, x_n]$ is a polynomial ring. 
Where specific conditions for $K$ are required, we will note it as necessary.

%%%%%%%%%%%%%
%%%%%%%%%%%%%
\medskip
\subsection{General Results}

We start this subsection by recalling the notion of $\Supp(f)$ 
(see for instance~\cite{KR1} Definition 1.1.11).
Let $f \in P$ and let $f= \sum_{i=1}^r c_it_i$
where $c_i \in K$ and $t_i\in \mathbb T^n$. Then the support of $f$ is
defined as $\Supp(f) = \{t_i \mid c_i \ne 0\}$. Notice that $\Supp(0) = \emptyset$.

\begin{definition}\label{factorclosedpoly}
A polynomial $f \in P$ is called \textbf{factor-closed} if there exists $t \in \Supp(f)$
such that  all  $t' \in \Supp(f)$ have the property that $t'$ divides $t$.
\end{definition}

\begin{lemma}\label{minimalGbasis}
Let $I\subset P$ be an ideal. Let $\sigma$ be a term ordering and $G$ be a minimal monic $\sigma$-Gr\"obner basis of $I$. Assume that
every polynomial in $G$ is factor-closed.
\begin{enumerate}
\item[(a)] The set $G$ is the reduced $\sigma$-Gr\"obner basis of $I$.

\item[(b)]  We have $\GFN(I) = 1$.
\end{enumerate}
\end{lemma}

\begin{proof}
Let us prove claim (a). For contradiction assume that $G$ is not reduced.
Since it is minimal and monic, there exist $i ,j \in \{1, \dots s\}$ and a
power product $\tilde{t}\in \Supp(g_i)$ such that $\LT_\sigma(g_j)\mid \tilde{t}$.
Since $g_i$ is factor-closed we deduce that  $\LT_\sigma(g_j) \mid  \LT_\sigma(g_i$),
a contradiction to the minimality of $G$.

The proof of (b) follows from the observation that
for every $i \in \{1,\dots, s\}$ the leading term of $g_i$
is the same for every  term ordering,
hence $G$ is the reduced Gr\"obner basis of $I$ for every term ordering.
\end{proof}

\begin{theorem}\label{GFan=Basic}
Let $I\subset P$ be an ideal. The following conditions are equivalent.
\begin{enumerate}
\item[(a)] There exists a term ordering $\sigma$ and a minimal monic
$\sigma$-Gr\"obner basis $G$ of $I$ such that
all the polynomials in $G$ are factor-closed.

\item[(b)]  There exists a term ordering $\sigma$ such that
all the polynomials in the reduced $\sigma$-Gr\"obner basis
of $I$ are factor-closed.

\item[(c)]  We have $\GFN(I) = 1$.
\end{enumerate}
\end{theorem}

\begin{proof}
From Lemma \ref{minimalGbasis}, we deduce that Claims (a) and (b) are equivalent and that (a) $\Rightarrow$ (c).
Next we prove (c) $\Rightarrow$ (b). By contradiction we assume that there
exists $i$ and a power product $\tilde{t} \in \Supp(g_i)$ such that $\tilde{t}$ does not
divide $\LT_\sigma(g_i)$. We let $t'= \tilde{t}/\gcd(\tilde{t},\LT_\sigma(g_i))$ and
$t = \LT_\sigma(g_i)) /\gcd(\tilde{t},\LT_\sigma(g_i))$. Then $t'$ and $t$ are
coprime and $t' \ne 1$.
Therefore there exists $x_j$ such that $x_j\,|\, t'$ and $x_j \nmid t$.
Let $\tau$ be the lexicographic term ordering with $x_j >_\tau x_i$ for $i \ne j$.
Then $\tilde{t} >_\tau \LT_\sigma(g_i)$ and hence the reduced $\tau$-Gr\"obner basis of $I$
is different from $G$. This is  a contradiction and the proof is complete.
\end{proof}

In the recent preprint~\cite{GG},  related results are proved for so-called neural ideals, which are generated by certain factor-closed generalizations of monomials (pseudomonomials) in Boolean rings.

Theorem \ref{GFan=Basic} gives an efficient way to check whether an ideal $I$ has a unique reduced Gr\"obner basis: in fact, one can simply inspect each minimal generator for being factor-closed.  This theorem also provides interesting consequences, as described in the following corollaries.

\begin{corollary}\label{oneGB-oneB}
Let $I$ be an ideal in  $P$ with $\GFN(I) = 1$,
and let $\mathcal{O}(I)$ be the unique G-basic set for $I$.
Then  $\mathcal{O}(I)$ is also the unique basic set for~$I$.
\end{corollary}

\begin{proof}
Let $G = \{g_1, \dots, g_s\}$ be the unique reduced Gr\"obner basis of $I$.
For contradiction, assume that there exists a basic set $\mathcal{O}$ for $I$ such that
$\mathcal{O} \ne \mathcal{O}(I)$, and let $t \in \mathcal{O} {\setminus} \mathcal{O}(I)$.
By definition of Gr\"obner basis, there exists $i$ such that $\LT(g_i)\, |\, t$.
From the theorem we know
that $g_i$ is factor-closed, hence every power product in $\Supp(g_i)$ divides~$t$.
On the other hand $\mathcal{O}$ is an order ideal, hence every power product in $\Supp(g_i)$
is in $\mathcal{O}$. Therefore we get a non-trivial linear combination of elements
of $\mathcal{O}$ which is zero in $P/I$, thus a contradiction.
\end{proof}

\begin{corollary}\label{monomial-oneB}
Let  $\Phi$ be a linear shift of~$P$. Let $I$ be a monomial ideal in $P$, and $\mathcal{O}(I)$  the set of power products
which are not divisible by any power product in $I$. 
\begin{enumerate}
\item[(a)] We have  $\GFN(I) =1$ and  $\mathcal{O}(I)$
is the unique basic set for $I$.

\item[(b)]  We have  $\GFN(\Phi(I)) =1$ and $\mathcal{O}(I)$
is the unique basic set for $\Phi(I)$.
\end{enumerate}
\end{corollary}
\begin{proof}
To prove Claim (a), we observe that
$\mathcal{O}(I)$ is the unique G-basic set for~$I$ by Theorem~\ref{GFan=Basic}.
Then the conclusion follows from Corollary~\ref{oneGB-oneB}.

Claim (b) follows from (a) and Theorem~\ref{linshifts}.
\end{proof}

When $I$ has a unique reduced Gr\"obner basis, the above results show that linear shifts preserve leading terms as well as basic sets.  In the remainder of this section and in Section 4, we will see similar results for other types of ideals.

We observe that a linear shift is composed of two types of shifts, namely $\Phi_1$ of type
$x_i\mapsto a_ix_i$ and $\Phi_2$ of type $x_i \mapsto x_i+b_i$.
Clearly, if $I$ is a monomial ideal we have $I = \Phi_1(I)$, so the only non-trivial part of
Corollary~\ref{monomial-oneB} is that $\GFN(\Phi_2(I)) =1$.

%%%%%%%%%%%%%
%%%%%%%%%%%%%
\medskip
\subsection{Distractions}
In Corollary~\ref{monomial-oneB} we have seen a modification of monomial ideals
which produces ideals with GFan number equal to 1.  In the literature there is another
interesting  construction which yields the same result. For a complete introduction to the theory
of distractions, see~\cite{KR2}.

\begin{definition}\label{defofdidtrac}
Let~$K$ be an infinite field. For $i\!=\!1, \dots, n$, 
let ${\pi_i = (c_{i1},c_{i2},\dots)}$ be a sequence with $c_{ij}\in K$ and
$c_{ij}\ne c_{ik}$ for every $j\ne k$. Set
$\pi=(\pi_1,\dots,\pi_n)$.

\begin{enumerate}
\item For every power product $t=x_1^{\alpha_1}\cdots x_n^{\alpha_n}$
in $\mathbb T^n$, the  polynomial
$${
D_\pi(t) = \tprod_{i=1}^{\alpha_1} (x_1-c_{1i}) \;\cdot\;
\tprod_{i=1}^{\alpha_2} (x_2-c_{2i}) \;\cdots\;
\tprod_{i=1}^{\alpha_n} (x_n-c_{ni})
}$$
is called the \textbf{distraction} of~$t$ with respect to~$\pi$.
\index{distraction!of a term}%

\item Let $I$ be a monomial ideal in~$P$, and let $\{t_1,\dots,t_s\}$
be the unique minimal monomial system of generators of~$I$. Then
we say that the ideal ${D_\pi(I)=\langle D_\pi(t_1),\dots,D_\pi(t_s)\rangle}$ is
the \textbf{distraction} of~$I$ with respect to~$\pi$.
\end{enumerate}
\end{definition}

\begin{theorem}[\cite{KR2}]\label{6.2.12}
Let $I$ be a monomial ideal in $P$, let $\{t_1, \dots, t_s\}$ be a minimal set of
power products which generates $I$, let $\pi=(\pi_1, \dots, \pi_n)$ be sequences of pairwise
distinct elements in $K$,
and  let $D_\pi(I) =\langle D_\pi(t_1), \dots, D_\pi(t_s)\rangle$ be
the corresponding distraction of $I$.
\begin{enumerate}
\item[(a)] The ideal $D_\pi(I)$ is radical.

\item[(b)] The set $\{ D_\pi(t_1), \dots, D_\pi(t_s) \}$  is the reduced $\sigma$-Gr\"obner
basis of $D_\pi(I)$ for every term ordering $\sigma$.

\item[(c)] We have $\GFN(D_\pi(I)) = 1$.
\end{enumerate}
\end{theorem}

See Theorem 6.2.12 in~\cite{KR2} for a proof.

\begin{example}\label{firstdistraction}
Let $K = \mathbb Q$.  Consider $I = \langle t_1, t_2\rangle$, where  $t_1 = x^3y$ and $t_2 = x^2y^4$.  
Set $\pi_1 = (3,2,5)$, $\pi_2 = (2, -1,3,12)$. As the elements in $\pi_1, \pi_2$
are pairwise distinct, we can make the distraction of $I$ with respect to $\pi= (\pi_1, \pi_2)$:
$D_\pi(I) = \langle D_\pi(t_1), D_\pi(t_2)\rangle$ where
$$D_\pi(t_1) = (x-3)(x-2)(x-5)(y-2),$$
%and 
$$D_\pi(t_2) = (x-3)(x-2)(y-2)(y+1)(y-3)(y-12).$$
According to Theorem \ref{6.2.12}, the ideal $D_\pi(I)$ is radical and$\{D_\pi(t_1), D_\pi(t_2)\}$ 
is the reduced
$\sigma$-Gr\"obner basis for every $\sigma$,
and so $\GFN(D_\pi(I)) =~1$.
\end{example}

The assumption that $K$ is infinite guarantees the existence 
of a distraction of all monomial ideals.
However, in order to define the distraction~$D_\pi(I)$ of a single
monomial ideal~$I$, it suffices to specify
the first $d_i$  elements of the sequence~$\pi_i$
where $d_i = \max\{\deg_{x_i}(t_j)\mid j\in\{1,\dots,s\}\}$ for $i =1,\dots, n$.
In particular, to distract a monomial ideal $I$,
it is sufficient to use finite tuples of elements.
Consequently we do not have to assume that~$K$ is infinite,
as long as $K$ has \textit{sufficiently many} elements.

\begin{example}
Consider the ideal in Example~\ref{firstdistraction}.
As the largest exponent is~4, we need a field which  has at least four elements.
So $\mathbb F_2$ and $\mathbb F_3$ are excluded.
On the other hand, if $K = \mathbb F_5$ we can choose $\pi=(\pi_1, \pi_2)$
where  $\pi_1 = (1,3,0)$, $\pi_2 = (0,1,2,3)$. Then we get  $D_\pi(t_1) = (x-1)(x-3)xy$, \
$D_\pi(t_2) = (x-1)(x-3)y(y-1)(y-2)(y-3)$.
\end{example}

As a consequence of Theorem~\ref{6.2.12} and Proposition~\ref{linshifts},
we get the following result.

\begin{corollary}\label{distrpoints}
We make the same assumptions as in Theorem~\ref{6.2.12} with the extra-condition 
that $I$ is a zero-dimensional ideal. 

\begin{enumerate}
\item[(a)] The ideal $D_\pi(I)$ is an ideal of points and
$\GFN(D_\pi(I)) = 1$.

\item[(b)]  If $\Phi$ is a linear shift of $P$, the ideal $\Phi(D_\pi(I))$ is an ideal of points, and
$\GFN(\Phi(D_\pi(I))) = 1$.
\end{enumerate}
\end{corollary}

\begin{proof} Claim (a) follows from~\cite{KR2}, Theorem 6.2.12(a) and Theorem~\ref{6.2.12}(c).

Claim~(b) follows from (a) and Proposition~\ref{linshifts}.
\end{proof}

The following  examples illustrate interesting outcomes of this corollary.

\begin{example}\label{ex:variousdistractions}
%If $K = \mathbb Q$, then $f_1=x(x-\frac{1}{5})(x-2)(x+1)$ is the distraction of $x^4$ with respect to any permutation of the $4$-tuple $(0, \frac{1}{5}, 2, -1 )$. Likewise the polynomial $f_2=y(y-1)(y-2)$ is the distraction of $y^3$ with respect to any permutation of the $3$-tuple $(0, 1,2)$. If we add $x^2y$ and $xy^2$ to the monomial ideal generated by the set $\{x^4, y^3\}$, we get the monomial ideal $J=\langle x^4,\, y^3,\, x^2y,\, xy^2\rangle$. A distraction of $J$ is  $I_1= \langle f_1, f_2, f_3, f_4 \rangle$ where $f_3 = x(x-\frac{1}{5})y$, $f_4 = xy(y-1)$. In this case $\pi_1 = (0, \frac{1}{5},2, -1)$ and $\pi_2 = ( 0, 1, 2)$. According to the corollary, $I_1$ is an ideal of points. More precisely we have $I_1 = I(\mathbb X)$ where $\mathbb X = \{ (0,0), (\frac{1}{5} ,0), (0,1), (\frac{1}{5},1), (0,2) \}$. See the following  \cocoa-5\ session.
%
%\medskip
Consider the monomial ideal $J=\langle x^4,\, y^3,\, x^2y,\, xy^2\rangle\subset \mathbb Q[x,y]$.  We will show how to construct a set of points $\mathbb X$ such that its ideal of points~$\mathcal I(\mathbb X)$ is a distraction of $J$. 

Since the first two generators of $J$ are not mixed and have degrees larger than the powers of $x$ and $y$ in the other two generators, we can use $x^4$ and $y^3$ to construct two sequences and two polynomials.    Let $\pi_1$ be any sequence with at least 4 entries, say $(0, \frac{1}{5}, 2, -1,\ldots )$.  

Then the polynomial $f_1:=x(x-\frac{1}{5})(x-2)(x+1)$ is the distraction of $x^4$ with respect to $\pi_1$.  Similarly let $\pi_2$ be any sequence with at least 3 entries, say $(0, 1, 2, \ldots )$.  Then $f_2:=y(y-1)(y-2)=D_{\pi_2}(y^3)$.  Set $\pi=(\pi_1,\pi_2)$.  Now we can construct the distractions of the other two power products, namely $f_3 := x(x-\frac{1}{5})y$ and $f_4 := xy(y-1)$.  

Consider the ideal $I= \langle f_1, f_2, f_3, f_4 \rangle$.  Notice that $I$ is the distraction of $J$ with respect to $\pi$.  Furthermore, $I$ is the ideal of the points 
$$\mathbb X = \left\{ (0,0), (0,1), (0,2), \left(\tfrac{1}{5} ,0\right), \left(\tfrac{1}{5},1\right), (2,0), (-1,0) \right\}.$$  
We observe that $\mathcal I(\mathbb X)=D_{\pi}(J)$ and  $\GFN(D_{\pi}(J))=1$ follows from Corollary~\ref{distrpoints}(a).  

%See the following  \cocoa-5\ session.
%
%\scriptsize{
%\begin{verbatim}
%Use P::=QQ[x,y];
%A1:=x; A2:= x-2; A3:= x+1; A4:=x-(1/5);
%B1:=y; B2:= y-1; B3:= y-2;
%F1:= A1*A2*A3*A4; LT(F1);
%F2:=B1*B2*B3; LT(F2);
%F3:= A1*A4*B1; LT(F3); 
%F4:= A1*B1*B2; LT(F4);
%D1:=Ideal(F1,F2,F3,F4);
%PD1:=PrimaryDecomposition0(D1); [ReducedGBasis(X) | X In PD1];
%-- [[y, x], [y -1, x], [y -2, x], [y, x +1], [y, x -2], [y, x -1/5], [y -1, x -1/5]]
%
%\end{verbatim}
%}
%\normalsize
\end{example}

The following examples show that if we do not follow the rigid order in the choice of the constants 
imposed by the definition of distraction, unexpected things can happen.

\begin{example}\label{ex:wrong-order}
If we consider the polynomials $f_1, f_2$ of Example~\ref{ex:variousdistractions}
and the two polynomials 
$\ell_3 = (x-2)(y-1)(y-2),\quad
\ell_4 = (x+1)\left(x-\frac{1}{5}\right)(y-1)$, then
the ideal $I_4 = \langle f_1, f_2, \ell_3, \ell_4 \rangle$ is not a distraction of $J$ 
for any permutation of the tuples $(0, \frac{1}{5}, 2, -1)$ and $(0,1,2)$.
However, it has the unique reduced Gr\"obner basis 
$G=\{ x^4 -\tfrac{6}{5}x^3 -\tfrac{9}{5}x^2 +\tfrac{2}{5}x,   \
y^2 -3y +2,\  \ x^2y -x^2 +\tfrac{4}{5}xy -\tfrac{4}{5}x -\tfrac{1}{5}y +\tfrac{1}{5} \}$.

From the equalities
$$
\begin{array}{lll}
x^4 -\frac{6}{5}x^3 -\frac{9}{5}x^2 +\frac{2}{5}x &=& (x +1)(x -\frac{1}{5})(x)(x -2) \cr
y^2 -3y +2 &=& (y-1)(y-2)\cr
x^2y -x^2 +\frac{4}{5}xy -\frac{4}{5}x -\frac{1}{5}y +\frac{1}{5} &=& (y -1)(x +1)(x -\tfrac{1}{5}).
\end{array}
$$
we see that  $I_4$ is the distraction of the monomial ideal $\langle x^4,\, y^2,\, x^2y\rangle$ 
with respect to~$\pi = (\pi_1, \pi_2)$ 
where $\pi_1 = (-1,\, \tfrac{1}{5}, 0, 2)$ and $\pi_2 = (1,2)$.
\end{example}

\begin{example}\label{ex:newExample}
Consider the ideal $I$ of the following set of four points 
$$\{(0,0,0),\  (1,0,0),\  (1,1,0),\  (1,1,1)\}.$$
Its reduced Gr\"obner basis with respect to $\sigma={\tt DegRevLex}$ is
$$ \{z^2 -z,\  yz -z,\  xz -z, \ y^2 -y, \ xy -y, \ x^2 -x\}.$$\
All polynomials in this basis are factor-closed, hence we have $\GFN(I) = 1$ by 
Theorem~\ref{GFan=Basic}.
We have  $\LT_\sigma(I) =\langle z^2,\  yz,\  xz , \ y^2 , \ xy, \ x^2\rangle$.
However, $\LT_\sigma(I)$ is not a distraction of $I$ since 
 we have the equalities 
 $$yz-z = z\mathbf{(y-1)}\  \text{\  and \ }  \ xy-y =(x -1)\mathbf{y}$$

\end{example}

%%%%%%%%%%%%%
%%%%%%%%%%%%%
\medskip
\subsection{Natural Distractions and Staircases}
\label{Special Distractions}

In this subsection we introduce an interesting family of distractions.
We recall that an ideal is called \textit{irreducible} if it cannot be
written as the intersection of two ideals, both of which properly contain it,
and use some results from~\cite{KR2}.

\begin{proposition}\label{intersectirred}
Let $K$ be a field and let $P = K[x_1, \dots, x_n]$.
\begin{enumerate}
\item[(a)] Every proper ideal in $P$ is a finite intersection of irreducible ideals.

\item[(b)] A monomial ideal $I$ in $P$ is irreducible if and only if it is of the
form $I = \langle x_{i_1}^{d_1},\dots, x_{i_s}^{d_s}\rangle$ with $1 \le i_1 < \cdots < i_s\le n$
and $d_1, \dots, d_s \in \mathbb N$.

\item[(c)] A zero-dimensional monomial ideal is irreducible if and only if it is of the
form  $I = \langle x_1^{d_1},\dots, x_n^{d_n}\rangle$ with $d_1, \dots, d_s \in \mathbb N$.
\end{enumerate}
\end{proposition}

\begin{proof}
For Claim (a) see~\cite{KR2}, Proposition 5.6.17. For Claim (b) see~\cite{KR2}, Proposition 6.2.11.
Claim (c) follows immediately from (b).
\end{proof}

From Corollary~\ref{monomial-oneB} we know that associated to every
monomial ideal there is a unique basic set $\mathcal{O}(I)$, the set of
power products which are not divisible by any power product in $I$.
This observation motivates the following definition.
Since the exponents of every monomial are natural numbers, they can be viewed
as elements of any field $K$ via the natural map $\mathbb N \to K$.

\begin{definition}\label{associatedpoints}
Let $K$ be a field and let $P = K[x_1, \dots, x_n]$.

\begin{enumerate}
\item Given $t = x_1^{a_1}\cdots x_n^{a_n}$, we say that $p(t) :=
(a_1, \dots, a_n) \in K^n$ is the \textbf{point associated to $t$}.

\item Let $I$ be a monomial ideal and let $\mathcal{O}(I)$ be the unique basic set
associated to $I$. Then the set $\{ p(t) \ | \  t \in \mathcal{O}(I) \}$ is called the
\textbf{staircase} of points associated to $I$ and denoted by $\Stair(I)$.
\end{enumerate}
\end{definition}

\begin{example}\label{exstaircase}
Let $I = \langle x^2, \, xyz^2,\,  y^2, \, z^3\rangle \subset \mathbb Q[x,y,z]$. Then we have
$$
\mathcal{O}(I) = \{ 1,\,  z,\,  z^2, \, y, \, yz, \, yz^2,\,  x,\,  xz,\,  xz^2, \, xy, \, xyz \}.
$$
Hence we get
\begin{eqnarray*}
\Stair(I) = \{ (0,0, 0), \, (0,0, 1), \, (0,0, 2), \, (0,1, 0), \, (0,1,1) \, (0,1, 2),  & \\
(1,0, 0), \, (1,0, 1), \, (1,0, 2), \, (1,1, 0), \, (1,1, 1) \}.
\end{eqnarray*}
\end{example}

\begin{lemma}\label{intersectionmonomial}
Assume $I_1$, $I_2$ are zero-dimensional monomial ideals in $P$.
Let $\pi_1,  \dots, \pi_n$ be sequences of pairwise distinct elements of the field~$K$,
and set $\pi=(\pi_1, \dots, \pi_n)$.
\begin{enumerate}
\item[(a)] $D_\pi(I_1\cap I_2) = D_\pi(I_1) \cap D_\pi(I_2)$.

\item[(b)] $\mathcal{O}(I_1+I_2) = \mathcal{O}(I_1)\cap \mathcal{O}(I_2)$

\item[(c)] $\mathcal{O}(I_1\cap I_2) = \mathcal{O}(I_1)\cup \mathcal{O}(I_2)$.

\item[(d)] $\Stair(I_1\cap I_2) = \Stair(I_1)\cup \Stair(I_2)$.
\end{enumerate}
\end{lemma}

\begin{proof}
Claim (a) follows from~\cite{KR2}, Proposition 6.2.10.

To prove Claim (b), notice that the inclusions $I_1 \subseteq I_1+I_2$
and  $I_2 \subseteq I_1+I_2$ imply  $\mathcal{O}(I_1+I_2) \subseteq \mathcal{O}(I_1)$
and $\mathcal{O}(I_1+I_2) \subseteq \mathcal{O}(I_2)$.  It follows that 
$\mathcal{O}(I_1+I_2) \subseteq \mathcal{O}(I_1)\cap \mathcal{O}(I_2)$.
On the other hand, if $t$ is a power product with $t \notin I_1$ and $t \notin I_2$,
then $t \notin I_1+I_2$ and the claim is proved.

Let us now prove (c). From the inclusions $I_1\cap I_2 \subseteq I_1$
and  $I_1\cap I_2 \subseteq I_2$ we get the inclusions $\mathcal{O}(I_1) \subseteq \mathcal{O}(I_1\cap I_2)$
and $\mathcal{O}(I_2) \subseteq \mathcal{O}(I_1\cap I_2)$, hence the inclusion
$\mathcal{O}(I_1)\cup \mathcal{O}(I_2) \subseteq \mathcal{O}(I_1\cap I_2)$.
To conclude the proof, we need to show that the two sets have the same number of elements.
On the other hand, if $I$ is a zero-dimensional monomial ideal,
the number of elements of $\mathcal{O}(I)$ is finite. Since  we have
$$\card\big(\mathcal{O}(I_1)\cup \mathcal{O}(I_2)\big) = \card(\mathcal{O}(I_1)) +\card(\mathcal{O}(I_2))
- \card\big(\mathcal{O}(I_1)\cap \mathcal{O}(I_2)\big) $$
we need to prove the equality
$$
\card( \mathcal{O}(I_1\cap I_2)) =\card(\mathcal{O}(I_1)) +\card(\mathcal{O}(I_2))
- \card\big(\mathcal{O}(I_1)\cap \mathcal{O}(I_2)\big). \eqno{(1)}
$$
To show this equality, we construct the exact sequence of $K$-vector spaces
$$
0 \;\to\; P/(I_1\cap I_2) \;\to\; (P/I_1) \oplus (P/I_2) \;\to\; P/(I_1+I_2) \;\to\; 0
$$
defined by the map
$$P/(I\cap J) \to (P/I)\oplus (P/J) \text{  given by  } f+(I\cap J) \mapsto (f+I,f+J)$$
and the map
$$(P/I)\oplus (P/J)\to P/(I+J) \text{  given by  } (f+I,\, g+J) \mapsto f-g+I+J.$$
From the exact sequence, we get the equality
$$
\card( \mathcal{O}(I_1\cap I_2)) = \card(\mathcal{O}(I_1)) +\card(\mathcal{O}(I_2))
- \card\big(\mathcal{O}(I_1+I_2)\big).
$$
From Claim (b) we deduce that this equality coincides with  (1).

Claim (d) follows immediately from (c)
and the definition of a staircase. 
Hence the proof is complete.
\end{proof}

We are ready to introduce a  special type of distractions.

\begin{definition}
Consider a monomial ideal $I$ in~$P$.  Let $d_1, \dots, d_n$ be
the maximal exponents of $x_1, \dots, x_n$ in the minimal set of generators of~$I$ and take $d = \max\{d_i\}$. Assume that $0, 1, \dots, d-1$ are distinct elements of $K$.
Furthermore, let $\pi_{d_i}= (0,1,2,3,\dots, d_i-1)$ and let $\pi_{\rm nat} = (\pi_{d_1}, \dots, \pi_{d_n})$.
Then the distraction $D_{\pi_{\rm nat}}(I)$ is called the \textbf{natural  distraction} (or
classic distraction) of the ideal~$I$.
\end{definition}

The following proposition shows the connections between natural distractions and staircases.

\begin{proposition}\label{naturaldistraction}
Consider a monomial ideal $I$ in~$P$.  
Then we have the equality $\mathcal I(\Stair(I)) = D_{\pi_{\rm nat}}(I)$.
\end{proposition}

\begin{proof}
As a first step, we prove the claim with the extra assumption that  $I$ is irreducible, hence of type
$I= \langle x_1^{a_1}, \dots, x_n^{a_n}\rangle$. In this case it is easy to see that
$\Stair(I) = \{(c_{1}, \dots, c_{n})\ | \  0\le c_{k}<a_k \hbox{\ for\ } k= 1, \dots, n \}$.
Consequently we have
\begin{eqnarray*}
\mathcal{I}(\Stair(I)) &=& \bigcap_{0\le c_{k}<a_k} \langle x_1-c_{1}, \dots, x_n-c_{n}\rangle\\
 &=&
\left\langle \prod_{c_1=1}^{a_1-1}(x_1-c_{1}), \dots, \prod_{c_n=1}^{a_n-1}(x_n-c_{n}) \right\rangle\\
&=& D_{\pi_{\rm nat}}(I).
\end{eqnarray*}
Next we prove the general claim. From Proposition~\ref{intersectirred} we get an equality $I = \bigcap_{k=1}^sJ_k$
with $J_k$ irreducible for $k =1, \dots, s$. We deduce the following equalities
\begin{eqnarray*}
\mathcal I(\Stair(I))\ = \  \mathcal I\left(\Stair\left(\cap_{k=1}^s(J_k)\right)\right)\  \Eqq{(1)}\
\mathcal I(\cup_{k=1}^s\Stair(J_k)) = \\
 \bigcap_{k=1}^s \mathcal I(\Stair(J_k))\  \Eqq{(2)}
 \bigcap_{k=1}^s D_{\pi_{\rm nat}}(J_k) \  \Eqq{(3)}  D_{\pi_{\rm nat}}(\cap_{k=1}^sJ_k) = D_{\pi_{\rm nat}}(I)
\end{eqnarray*}
where Equality (1) follows from Lemma~\ref{intersectionmonomial}(d), Equality (2) follows
from the special case discussed above, and Equality (3) follows from Lemma~\ref{intersectionmonomial}(a).
\end{proof}

The following example illustrates the special feature of natural distractions proved in the theorem
 above.
 
\begin{example}\label{exstaircasedistr}
Let $P = \mathbb Q[x, y]$ and let $I = \langle x^5,\, x^4y,\, xy^2,\, y^4 \rangle$.
Then we have
\begin{eqnarray*}
D_{\pi_{\rm nat}}(I) = \langle\ x(x-1)(x-2)(x-3)(x-4),\  x(x-1)(x-2)(x-3)y, \\
 xy(y-1),\  y(y-1)(y-2)(y-3)\ \rangle.
    \end{eqnarray*}
If we draw a picture of the power products involved, we get
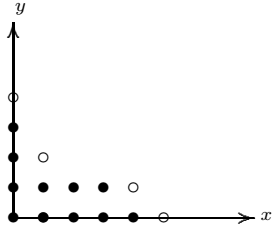
\bigbreak
\makebox[11 true cm]{

\beginpicture
\setcoordinatesystem units <0.4cm,0.4cm>
\setplotarea x from 0 to 8, y from 0 to 6.5
\axis left /
\axis bottom /

\arrow <2mm> [.2,.67] from  7.5 0  to 8 0
\arrow <2mm> [.2,.67] from  0 6  to 0 6.5

\put {$\scriptstyle x$} [lt] <0.5mm,0.8mm> at 8.1 0
\put {$\scriptstyle y$} [rb] <1.7mm,0.7mm> at 0 6.6
\put {$\bullet$} at 0 0
\put {$\bullet$} at 1 0
\put {$\bullet$} at 0 1
\put {$\bullet$} at 2 0
\put {$\bullet$} at 1 1
\put {$\bullet$} at 1 0
\put {$\bullet$} at 0 2
\put {$\bullet$} at 3 0
\put {$\bullet$} at 2 1
\put {$\bullet$} at 0 3
\put {$\bullet$} at 4 0
\put {$\bullet$} at 3 1
%\put {$\scriptstyle 1$} [lt] <-0.6mm,-1mm> at 0 0
\put {$\circ$} at 0 4
\put {$\circ$} at 1 2
%\put {$\circ$} at 1 3
%\put {$\circ$} at 2 2
%\put {$\circ$} at 3 2
\put {$\circ$} at 4 1
\put {$\circ$} at 5 0
%\put {$\scriptstyle\times$} at 0 5
%\put {$\scriptstyle\times$} at 1 4
%\put {$\scriptstyle\times$} at 2 3
%\put {$\scriptstyle\times$} at 3 3
%\put {$\scriptstyle\times$} at 4 2
%\put {$\scriptstyle\times$} at 5 1
%\put {$\scriptstyle\times$} at 6 0

\endpicture}

\smallskip
\noindent
where the white dots represent the generators of $I$ and
the black dots represents the power products in $\mathcal{O}(I)$.
According to Proposition~\ref{naturaldistraction}, we can check that
the set of points defined by the ideal $D_{\pi_{\rm nat}}(I)$
is exactly the staircase represented by the black dots.
\end{example}

%%%%%%%%%%%%%%%%%%%%%%%%%%%%%%%%%%%%
%%%%%%%%%%%%%%%%%%%%%%%%%%%%%%%%%%%%
\medskip
\section{\large Complementary ideals}
\label{Complementary Ideals}

In this section we concentrate on zero-dimensional ideals and
introduce the notion of complementary ideals (see Definition~\ref{mainAssump2}). 
Throughout the section, we let $K$ be a field and $P=K[x_1, \dots, x_n]$ as before. 

%%%%%%%%%%%%%
%%%%%%%%%%%%%
\medskip
\subsection{Grids}
It is well-known that under the assumption that $I$ is a zero-dimensional ideal,
we have $\dim_K(P/I)<\infty$ (see for instance~\cite{KR1}, Proposition 3.7.1).
Consequently, for every $i =1, \dots, n$, the natural embedding
$K[x_i]/(I\cap K[x_i]) \hookrightarrow P/I$ shows that $I\cap K[x_i] \ne\langle 0\rangle$,
and hence $I\cap K[x_i] $ is a principal non-zero ideal. We denote
its monic generator by $f_I(x_i)$. These facts motivate the following definition.

\begin{definition}\label{maximalgrid}
Let $I$ be a zero-dimensional ideal in $P$ and   for $i =1, \dots, n$,
let $f_I(x_i)$ be the monic generator of $I \cap K[x_i]$.
\begin{enumerate}
\item The ideal $\langle f_I(x_1), \dots, f_I(x_n)\rangle \subseteq I$ is called 
the \textbf{maximal grid ideal contained in  $I$} and denoted by $\maxgrid(I)$.

\item More generally, every ideal  in $P$ of type
$\langle g_1(x_1), \dots, g_n(x_n)\rangle$ and such that  ${\deg(g_i(x_i)) >0}$
for every $i =1, \dots, n$ is called a \textbf{grid ideal}.

\end{enumerate}
\end{definition}

%\begin{remark}\label{gridmultiple}
When $I$ is a zero-dimensional ideal and 
$J = \langle g_1(x_1), \dots, g_n(x_n)\rangle$ is a grid ideal such that $J\subseteq I$,
 it is clear $g_i(x_i)$ is a multiple of $f_I(x_i)$ for every $i=1,\dots, n$.
This observation validates the use of ``maximal'' in the above definition.
%This is the reason why $\maxgrid(I)$ is called the maximal grid ideal contained in $I$.
%\end{remark}

Special cases of grid ideals are obtained as follows, using language from 
combinatorial experimental design.

\begin{definition}\label{grid}
Let $d_1, \dots, d_n \in \mathbb N_+$ and for $i=1,\dots,n$,
let $ (c_{i,1}, \dots, c_{i,d_i})$ be a $d_i$-tuple of pairwise distinct
elements of $K$.
Then the following set of points
$\mathbb X=\{ (c_{1,k_1},  \dots, c_{n, k_n}) \quad | \quad
1\le k_1\le d_1, \dots, 1\le k_n \le d_n\}$
is called a \textbf{grid of points} or a \textbf{full design} in $K^n$.
It consists of $\prod_{i=1}^nd_i$ points.
The vanishing ideal of $\mathbb X$ is a grid ideal generated by
$\{g_1(x_1),\dots,  g_n(x_n)\}$ where $g_i(x_i)= \prod_{j=1}^{d_i}(x_i-c_{ij})$.
\end{definition}

\begin{remark}\label{minimalgrid}
Let $\mathbb Y$ be a set of points.
The maximal grid ideal contained in~$\mathcal I(\mathbb Y)$ is the vanishing ideal of
a set $\mathbb X$ of points. In agreement with Definition~\ref{maximalgrid},
the set $\mathbb X$ is called the \textbf{minimal grid of points containing $\mathbb Y$}.
\end{remark}

\begin{example}\label{distrgrid}
As observed in Remark~\ref{minimalgrid},
given a set of points $\mathbb{Y} \subseteq K^n$, its vanishing ideal $\mathcal{I}(\mathbb{Y})$
contains $n$ univariate polynomials $f_i(x_i)$ which are products of linear
polynomials of type $x_i - c_{ij}$, and define the minimal
grid $\mathbb{K}$ containing~$\mathbb{Y}$.
If $d_i =\deg(f_i(x_i))$, then $f_i(x_i)$ is the distraction of $x_i^{d_i}$ with respect to
any permutation of the tuple of the $c_{ij}$'s.

\end{example}

We recall that for a zero-dimensional local ring $R$
with maximal ideal~$\m$,
the socle of $R$ is defined as $\soc(R) = \Ann_R(\m)$, the
annihilator of $\m$. The following easy lemma collects some properties
of  grid ideals.

\begin{lemma}\label{uniquebasis}
Let  $J=\langle g_1(x_1), \dots,g_n(x_n)\rangle$ be a grid ideal 
with $d_i = \deg(g_i(x_i))$ for $i = 1, \dots, n$.

\begin{enumerate}

\item[(a)] The ideal $J$ is zero-dimensional.

\item[(b)] The set $\{g_1(x_1), \dots, g_n(x_n)\}$
is the reduced $\sigma$-Gr\"obner basis  of~$J$
for every term ordering $\sigma$, and hence $\GFN(J) = 1$.

\item[(c)] For every term ordering $\sigma$ the residue class of
$ x_1^{d_1-1}x_2^{d_2-1}\cdots x_n^{d_n-1}$
generates $\soc(P/\LT_\sigma(J))$.

\item[(d)] We have $\mathbb{T}^n{\setminus} \LT_\sigma(J) =
\{ t \in \mathbb{T}^n \  | \  t \ {\rm divides}\   x_1^{d_1-1}x_2^{d_2-1}\cdots x_n^{d_n-1}\} $
for every term ordering $\sigma$ on $\mathbb T^n$. \end{enumerate}
\end{lemma}

\begin{proof}
Claim (a) follows from the Finiteness
Criterion (see~\cite{KR1}, Proposition 3.7.1).
Claim (b) follows from the fact that $\LT_\sigma(g_i) = x_i^{d_i}$,
hence they are pairwise coprime. The other claims follow immediately.
\end{proof}

This lemma suggests that 
$x_1^{d_1-1}x_2^{d_2-1}\!\cdots x_n^{d_n-1}$ will be denoted by~$t_{soc}(J)$.
From Lemma~\ref{uniquebasis}(b) we know that if $J$ is a grid ideal we have  $\GFN(J) = 1$.
Therefore the set $\mathcal{O}_\sigma(J)$ is the same for every $\sigma$ and  hence it will be 
denoted by~$\mathcal{O}(J)$.

\begin{definition}\label{mainAssump2}
Let $J$ be a grid ideal  and let $J=\q_1\cap\cdots\cap\q_s$ be its primary decomposition.
For $0< t <s$, 
set $I_1 = \q_1\cap\cdots \cap\q_t$,
and $I_2=\q_{t+1}\cap\cdots \cap\q_s$.
Then we say that $I_1$, $I_2$ are \textbf{complementary ideals with respect to $J$},
or simply complementary ideals, if $J$ is clear from the context.
\end{definition}

The following lemma collects some  properties of complementary ideals.

\begin{lemma}\label{IandJ}
Let $J$ be a grid ideal in $P$ and let $I_1$, $I_2$ be complementary ideals with respect to $J$.

\begin{enumerate}
\item[(a)] $J=I_1\cap I_2$, \  $I_1 + I_2= \langle 1\rangle$, \  $I_2 = J:I_1$, \
and \  $I_1 =J:I_2$.

\item[(b)] There is an isomorphism of $K$-algebras $\varphi: P/I \cong  P/I_1 \times P/I_2$.

\item[(c)] $\dim_K(P/J) = \dim_K(P/I_1) + \dim_K(P/I_2)$, and
hence we have \\  ${\rm card}(\mathcal{O}(J))= {\rm card}(\mathcal{O}_\sigma(I_1)) 
+ {\rm card}(\mathcal{O}_\sigma(I_2))$
for every term ordering $\sigma$.

\item[(d)] If $I$ is a zero-dimensional ideal in $P$, then Claims (a), (b), and (c) hold
for $J = \maxgrid(I)$, $I_1 = I$, and $I_2 = \maxgrid(I):I$.
\end{enumerate}
\end{lemma}

\begin{proof}
Claim (a) can be proved using standard facts in commutative algebra.
Claim (b) follows from (a) and the Chinese Remainder Theorem
(see for instance~\cite{KR1}, Lemma 3.7.4). Claim (c) follows from~(b) since
the residue classes of the elements of $B_\sigma(I)$
form a $K$-basis of $P/I$ for any term ordering $\sigma$
and any ideal $I$ in $P$. Finally, Claim (d) is a consequence of the fact that
$\maxgrid(I)$ is a grid ideal which contains $I$.
\end{proof}

%Let us see an example which illustrates this lemma.
%
%\begin{example}\label{ex:mgrid}
%\scriptsize{
%\begin{verbatim}
%Use P::= QQ[x,y,z];
%M:=mat([ [0,0,0], [1,1,1], [-1,-1,-1] ]);
%I:=IdealOfPoints(P,M);
%-- This is the ideal of the three points
%fx:=MinPolyQuot(x,I,x); fy:=MinPolyQuot(y,I,y); fz:=MinPolyQuot(z,I,z);
%-- These are the equations of the maximal grid ideal contained in I
%mgridI:=ideal(fx,fy,fz);
%ColI:=mgridI:I;
%mgridI = Intersection(I, ColI);
%-- true
%EmgridI:= Multiplicity(P/mgridI); EI:= Multiplicity(P/I); EColI:= Multiplicity(P/ColI);
%EJ;EI;EColI;
%/*
%27
%3
%24
%*/
%EmgridI= EI + EColI;
%-- true
%\end{verbatim}
%}
%\end{example}
%
The following result is one of the main contributions of this paper.

\begin{theorem}\label{sameGFan}
Let $J$ be a grid ideal in $P$,  let $I_1$, $I_2$ be complementary ideals
with respect to $J$, and let $\sigma$ be a term ordering.

\begin{enumerate}
\item[(a)]  For $i=1,2$ we have 
$\mathcal{O}(J) \cap  \LT_\sigma(I_i)  = \mathcal{O}(J){\setminus}\mathcal{O}_\sigma(I_i)$.

\item[(b)] Let $\alpha_\sigma: \mathcal{O}(J) \cap  \LT_\sigma(I_i) \to  \mathbb T^n$ 
be the map which sends $t\mapsto t_{soc}(J)/t$. Then~$\alpha_\sigma$ 
is injective and induces a map 
$\vartheta_\sigma: \mathcal{O}(J){\setminus}\mathcal{O}_\sigma(I_1) 
\to \mathcal{O}_\sigma(I_2)$ which is bijective.
%
%\item[(c)] 
%There is a bijection between the set of the G-basic sets for $I_1$
%and the set of the G-basic sets for $I_2$.

\item[(c)] We have ${\rm GFan}(I_1) = {\rm GFan}(I_2)$ 
and hence $\GFN(I_1) = \GFN(I_2)$.
\end{enumerate}
\end{theorem}

\begin{proof}
First, we prove Claim (a).  From $I_i\supseteq J$, we deduce
that $\LT(I_i) \supseteq \LT(J)$,
hence $\mathcal{O}_\sigma(I_i) \subseteq \mathcal{O}(J)$ which implies the claim.

Next we prove Claim (b). The fact that $\alpha$ is injective is by construction. 
By contradiction, assume that  $t_{soc}(J)/t \notin \mathcal{O}_\sigma(I_2)$.
Then $t_{soc}(J)/t \in \LT_\sigma(I_2)$.
We have
$t_{soc}(J) = t\cdot t_{soc}(J)/t \in \LT_\sigma(I_1)\cdot \LT_\sigma(I_2)
\subseteq \LT_\sigma(I_1\cdot I_2)$.
Clearly $I_1\cdot I_2 \subseteq J$, and so we
get $t_{soc}(J)\in \LT_\sigma(I)$ which
yields a contradiction by Lemma~\ref{uniquebasis}(b).
Consequently we get an injective map 
$\mathcal{O}(J) \cap  \LT_\sigma(I_1) \to \mathcal{O}_\sigma(I_2)$,
and from~(a) we can rewrite it as a map 
$\vartheta: \mathcal{O}(J){\setminus} \mathcal{O}_\sigma(I_1) \to \mathcal{O}_\sigma(I_2)$
which is injective. Lemma~\ref{IandJ}(c) shows that the two sets
have the same cardinality, hence we conclude that  $\vartheta$ is bijective.

%Next we  prove Claim (c). For $i=1,2$, we  let $\mathcal{G}_i$ 
%denote the set of G-basic sets for $I_i$, and let $\mathcal{H}_1$ denote the sets of type
%$\mathcal{O}(J){\setminus}\mathcal{O}_\sigma(I_1)$. 
%Then we let $\gamma: \mathcal{G}_1 \to \mathcal{H}_1$ be defined 
%by $\gamma(\mathcal{O}_\sigma(I_1)) = \mathcal{O}(J){\setminus}\mathcal{O}_\sigma(I_1)$
%which is clearly bijective, and we let 
%$\varrho: \mathcal{H}_1 \to \mathcal{G}_2$ be defined by 
%$\varrho\big( \mathcal{O}(J){\setminus}\mathcal{O}_\sigma(I_1)\big) = 
%\vartheta_\sigma(\mathcal{O}(J){\setminus}\mathcal{O}_\sigma(I_1)) = \mathcal{O}_\sigma(I_2)$.
%From the fact that $\vartheta$ is bijective we deduce that also $\varrho$ is bijective and the proof 
%of Claim (c) is complete.
%
%To prove Claim (d) we observe that the equality $\GFN(I_1) = \GFN(I_2)$ follows immediately
%from Claim (c). However, we can say more using Claim (b). 
We observe that 
the set $\mathcal{O}_\sigma(I_1))$ is the same for every
$\sigma$ which corresponds to a point inside a polyhedral cone of ${\rm GFan}(I_1)$.
Therefore also the set $\mathcal{O}(J){\setminus}\mathcal{O}_\sigma(I_1))$ is the same,
hence we deduce from Claim (b) that also the set $\mathcal{O}_\sigma(I_2))$ is the same, and the proof is complete.
\end{proof}

%let $\G_i$ be the set of reduced Gr\"obner bases of $I_i$ for $i=1,2$.
%We define a map $\alpha:\G_1 \to \G_2$ as follows. Let $\sigma$ be a term ordering and
%let $G_\sigma$ be the reduced $\sigma$-Gr\"obner basis of $I_1$. It identifies
%the set $B_\sigma(I_1)$ which, in turn, identifies  the set $B_\sigma(I_2)$
%via the map $\vartheta$. The set~$\mathcal{O}_\sigma(I_2)$ identifies $\LT_\sigma(I_2)$
%which, in turn, identifies the reduced $\sigma$-Gr\"obner basis of $I_2$
%by the uniqueness of the reduced $\sigma$-Gr\"obner basis of $I_2$.
%So $\alpha$ is defined.
%Interchanging the roles of $I_1$ and $I_2$ we define a map $\beta: \G_2 \to \G_1$
%in the same way, and it is clear from the construction that $\alpha$ and $\beta$
%are inverses to each other.

\begin{remark}\label{prodequalintersect}
In the proof of Claim (b), we used the fact that $I_1\cdot I_2 \subseteq J$.
Actually we have $I_1\cdot I_2 = J$ (see for instance~\cite{KR3}, Theorem 2.2.1).
\end{remark}

Let us illustrate the theorem with an example.

\begin{example}\label{ex:nonradical.cocoa5}
Let $P=\mathbb Q[x,y]$.  Consider the grid ideal 
$$J=\langle x(x^2+1)^2(x-1), (y^3-1)(y+2) \rangle$$ with primary decomposition 
$$\langle x,y +2\rangle \cap 
\langle x,y -1\rangle \cap 
\langle x, y^2 +y +1\rangle \cap 
\langle x-1,y +2\rangle \cap 
\langle x-1,y -1\rangle \cap 
\langle x -1, y^2 +y +1\rangle$$ 
$$ \cap \langle x^4 +2x^2 +1,y+2\rangle \cap
\langle x^4 +2x^2 +1,y-1\rangle \cap 
\langle x^4 +2x^2 +1,y^2 +y +1\rangle. 
$$
Let $I_1$ and $I_2$ are complementary with respect to $J$.  If 
$$I_1= 
\langle x,y +2\rangle \cap 
\langle x -1, y^2 +y +1\rangle \cap 
\langle x^4 +2x^2 +1,y+2\rangle,$$ 
then 
$$I_2=J:I_1=
\langle x, y -1\rangle \cap 
\langle x, y^2 +y +1\rangle \cap 
\langle x-1,y +2\rangle \cap 
\langle x-1, y -1\rangle \cap $$ 
$$\langle x^4 +2x^2 +1,y-1\rangle \cap 
\langle x^4 +2x^2 +1,y^2 +y +1\rangle. 
$$
We have $\GFN(I_1)=\GFN(I_2)=~2$.  See the following \cocoa-5  code for details.
\scriptsize{
\begin{verbatim}
K::=QQ; Use P::=K[x,y];
J:=ideal(x*(x^2+1)^2*(x-1), (y^3-1)*(y+2));
PD:=PrimaryDecomposition0(J);[ReducedGBasis(X) | X In PD];
/*
[[y +2, x], [y -1, x], [x, y^2 +y +1], [y +2, x -1],
[y -1, x -1], [x -1, y^2 +y +1], [y +2, x^4 +2*x^2 +1],
[y -1, x^4 +2*x^2 +1], [y^2 +y +1, x^4 +2*x^2 +1]]
*/
I1:= Intersection(ideal(x -1, y^2 +y +1), ideal(y +2, x),
ideal(y +2, x^4 +2*x^2 +1));
ReducedGBasis(I1); QB1:=QuotientBasisSorted(I1); QB1;
-- [x*y +2*x -y -2, y^3 +3*y^2 +3*y +2,
--  x^5 +2*x^3 +(4/3)*y^2 +x +(4/3)*y -8/3]
-- [1, y, x, y^2, x^2, x^3, x^4]
I2:=Intersection(ideal(x, y^2 +y +1), ideal(y -1, x -1),
    ideal(y -1, x), ideal(y +2, x -1), ideal(y -1, x^4 +2*x^2 +1),
    ideal(y^2 +y +1, x^4 +2*x^2 +1));
indent(ReducedGBasis(I2)); QB2:=QuotientBasisSorted(I2); QB2;
/*[
  y^4 +2*y^3 -y -2,
  x*y^3 -y^3 -x +1,
  x^5*y -x^5 +2*x^3*y -2*x^3 +(-4/3)*y^3 +x*y -x +4/3,
  x^6 -x^5 +2*x^4 -2*x^3 +x^2 -x
] */
-- [1, y, x, y^2, x*y, x^2, y^3, x*y^2, x^2*y, x^3, x^2*y^2,
--   x^3*y, x^4, x^3*y^2, x^4*y, x^5, x^4*y^2]
multiplicity(P/J); multiplicity(P/I1); multiplicity(P/I2);
-- 24
-- 7
-- 17
GF1:=GroebnerFanReducedGBases(I1);indent(GF1);
/*
  [x*y +2*x -y -2, y^3 +3*y^2 +3*y +2, x^5 +2*x^3 +(4/3)*y^2 +x +(4/3)*y -8/3],
  [x*y -y +2*x -2, y^2 +(3/4)*x^5 +(3/2)*x^3 +y +(3/4)*x -2,
          x^6 -x^5 +2*x^4 -2*x^3 +x^2 -x]
*/
GF2:=GroebnerFanReducedGBases(I2); indent(GF2);
/*
  [y^4 +2*y^3 -y -2, x*y^3 -y^3 -x +1,
        x^5*y -x^5 +2*x^3*y -2*x^3 +(-4/3)*y^3 +x*y -x +4/3,
        x^6 -x^5 +2*x^4 -2*x^3 +x^2 -x],
  [y^3 +(-3/4)*x^5*y +(-3/2)*x^3*y +(3/4)*x^5 +(-3/4)*x*y +(3/2)*x^3 +(3/4)*x -1,
       x^5*y^2 +2*x^3*y^2 +x^5*y +x*y^2 +2*x^3*y -2*x^5 +x*y -4*x^3 -2*x,
       x^6 -x^5 +2*x^4 -2*x^3 +x^2 -x]
*/\end{verbatim}
}
\end{example}

\begin{corollary}\label{IandJ2}
Let $J$ be a radical grid ideal in $P$ and let $I$ be an ideal in~$P$ such
that $I\supseteq J$.

\begin{enumerate}
\item[(a)]  The ideals $I$ and $J:I$ are radical.

\item[(b)] The ideals $I$ and $J:I$ are complementary ideals with respect to $J$, and
hence the conclusions of
Theorem~\ref{sameGFan} apply to $I_1=I$ and $I_2=J:I$.

\item[(c)] Let  $\mathbb X$ be a grid of points,  $J= \mathcal I(\mathbb X)$, and $I\supseteq J$. 
Then there exists a subset $\mathbb Y$ of $\mathbb X$
such that $I= \mathcal I(\mathbb Y)$, $J:I= \mathcal I(\mathbb X\setminus \mathbb Y)$,
and (a) and (b) are satisfied by the ideals $J$, $I$ and $J:I$.
\end{enumerate}
\end{corollary}

\begin{proof}
To prove Claim (a), notice that as the ideal $J$ is zero-dimensional and radical,  
there are maximal ideals $\m_1, \dots, \m_s$ in $P$ such that $J = \bigcap_{i=1}^s \m_i$.
Let $S =\{1, \dots, s\}$.  
The Chinese Remainder Theorem implies that there is an 
isomorphism $\varphi:P/J\cong \prod_{i\in S}P/\m_i$. Via this isomorphism, 
the image $\varphi(I)$ is a product of~$s$ ideals which are 
either $\langle 0\rangle$ or~$\langle 1 \rangle$.
Let $T\subseteq S$ be the subset of
indices which correspond to the zero ideals. Then $I = \bigcap_{i \in T}\m_i$ and 
hence it is radical. Similarly we find that $J:I$.

To prove Claim (b), it suffices to observe that we have
$J:I = \bigcap_{i \in S{\setminus}T}\m_i$.

Finally, to prove (c), let $\mathbb X$ be a grid of points in $K^n$, and for
$i=1, \dots, n$,   let $g_i = \prod_{j=1}^{d_i}(x_i-c_{ij})$. Then
the vanishing ideal of~$\mathbb X$ is~$\mathcal I(\mathbb X)=\langle g_1, \dots, g_n\rangle$.
The ideal $\mathcal I(\mathbb X)$ is  radical by construction, and every ideal
which contains $\mathcal I(\mathbb X)$ is the
vanishing ideal of a subset $\mathbb Y$ of $\mathbb X$,
\textit{i.e.}\ it is of type $\mathcal{I}(\mathbb {Y})$. Consequently
we have $\mathcal I(\mathbb X):\mathcal I(\mathbb {Y}) =
\mathcal{I}(\mathbb X{\setminus} \mathbb Y)$.
\end{proof}

\smallskip
The following example illustrates this corollary.

\begin{example}\label{radical}
\scriptsize{
\begin{verbatim}
K::=QQ; Use P::=K[x,y];
I:=ideal((x^2+1)*(x-1)*(x-2), (y^2-2)*(y+2));
J1:= I+ideal(x-1+y^2-2);
J2:=Colon(I,J1);

ReducedGBasis(J1); QB1:=QuotientBasis(J1); QB1;
-- [x -1, y^2 -2]
-- [1, y]
ReducedGBasis(J2); QB2:=QuotientBasis(J2); QB2;
-- [y^3 +2*y^2 -2*y -4,
     x^3*y +2*x^3 -2*x^2*y -4*x^2 +x*y +2*x -2*y -4,
     x^4 -3*x^3 +3*x^2 -3*x +2]
-- [1, y, y^2, x, x*y, x*y^2, x^2, x^2*y, x^2*y^2, x^3]
multiplicity(P/I); multiplicity(P/J1); multiplicity(P/J2);
-- 12
-- 2
-- 10
GF1:=GroebnerFanReducedGBases(J1);GF1;
-- [y^2 -2, x -1]
GF2:=GroebnerFanReducedGBases(J2);GF2;
-- [y^3 +2*y^2 -2*y -4,
--  x^3*y +2*x^3 -2*x^2*y -4*x^2 +x*y +2*x -2*y -4,
--  x^4 -3*x^3 +3*x^2 -3*x +2]
\end{verbatim}
}
\end{example}

%The implication of this result for experimental design is promising.  Suppose $K$ is a finite field and $V\subset K^n$ is a set of data points such that its ideal of points has $\GFN$ 1; that is, the data have a single associated order ideal\bs{lang and notation} $B(\mathcal I(V))$.  Then the set complement $K^n\setminus V$ is another data set whose ideal of points has $\GFN$ 1. Given one set of data with a unique associated monomial basis, we can construct a second data set with a unique (yet different) basis. As each basis represents a distinct set of monomials with which to write all models fitting the data, multiple bases mean multiple ways to represent such models with each representation providing a different set of predictions.  In this setting, uniqueness matters and is desired.   
%
%%%%%%%%%%%%%%%%%%%%%%%%%%%%%%%%%%%%
%%%%%%%%%%%%%%%%%%%%%%%%%%%%%%%%%%%%
\medskip
\section{\large Final Remarks}
\label{Application}

In this section we collect some consequences of the theoretical results
described in the preceding sections.
The significance of Theorem \ref{sameGFan} is the ability to 
quickly compute new ideals with $\GFN$ 1.  
This is especially convenient for network inference or design of 
experiments where data often have states in a finite field. 
Let $K$ be a finite field with characteristic $p>0$. Then $K$ is a finite-dimensional $\mathbb F_p$-vector space  hence
the number of its elements is $q=p^e$, where $e= \dim_{\,\mathbb F_p}(K)$.
Given an indeterminate $z$, the univariate polynomial  $z^q-z$ is called a \textbf{field equation}
of $K$ since  $z^q-z = \prod_{a \in K}(z-a)$ (see~\cite{J}, Section 4.13).
Consequently, if $P^{\mathstrut} = K[x_1, \dots, x_n]$ and $g_i = x_i^q-x_i$ for $i =1, \dots, n$, then the
ideal $\langle g_1, \dots, g_n\rangle$ is the vanishing ideal of a grid and hence
Corollary~\ref{IandJ2} applies to this case.

Let us see an example with $K = \mathbb F_3$.

\begin{example}\label{fieldeq}
\scriptsize{
\begin{verbatim}
K::=ZZ/(3); Use P::= K[x,y,z];
I:=ideal(x^3-x, y^3-y, z^3-z);
J1:= I+ideal(x^2-y-z);
J2:=Colon(I, J1);

ReducedGBasis(J1); QB1:=QuotientBasis(J1); QB1;
-- [y^2 -y*z +z^2 -y -z, x*y +x*z -x, x^2 -y -z, z^3 -z]
-- [1, z, z^2, y, y*z, y*z^2, x, x*z, x*z^2]

ReducedGBasis(J2); QB2:=QuotientBasis(J2); QB2;
-- [z^3 -z, y^3 -y, x*y^2 -x*y*z +x*z^2 +x*y +x*z,
         x^2*y +x^2*z +x^2 +y^2 -y*z +z^2 -1, x^3 -x]
-- [1, z, z^2, y, y*z, y*z^2, y^2, y^2*z, y^2*z^2, x, x*z,
        x*z^2, x*y, x*y*z, x*y*z^2, x^2, x^2*z, x^2*z^2]
multiplicity(P/I); multiplicity(P/J1); multiplicity(P/J2);
-- 27
--  9
-- 18
GF1:=GroebnerFanIdeals(J1);
GF2:=GroebnerFanIdeals(J2);
Len(GF1);Len(GF2);
-- 4
-- 4
GF1:=GroebnerFanReducedGBases(J1);indent(GF1);
/*
[
  [x^2 -y -z, z^3 -z, x*y +x*z -x, y^2 -y*z +z^2 -y -z],
  [x^2 -z -y, y^3 -y, x*z +x*y -x, z^2 -y*z +y^2 -z -y],
  [y -x^2 +z, x^3 -x, z^3 -z],
  [z +y -x^2, x^3 -x, y^3 -y]
]
*/
GF2:=(J2);indent(GF2);
/*
  [z^3 -z, y^3 -y, x*y^2 -x*y*z +x*z^2 +x*y +x*z,
     x^2*y +x^2*z +x^2 +y^2 -y*z +z^2 -1, x^3 -x],
  [y^3 -y, x^3 -x, x^2*z +x^2*y +z^2 +x^2 -y*z +y^2 -1,
     x*z^2 -x*y*z +x*y^2 +x*z +x*y, z^3 -z],
  [x^3 -x, z^3 -z, y^2 +x^2*y -y*z +x^2*z +z^2 +x^2 -1],
  [x^3 -x, z^2 -y*z +y^2 +x^2*z +x^2*y +x^2 -1, y^3 -y]
*/
\end{verbatim}
}
\end{example}

Theorem \ref{sameGFan} shows, among other results, that complementary
ideals have the same number of reduced Gr\"obner bases.
The advantage of this is that it may be computationally easy
to test whether a \emph{small} set of data has a unique
Gr\"obner basis associated to it and then to generate
a \emph{larger} set via the complement. Let us see an easy
application of this remark.

\begin{proposition}\label{gridminusgrid}
Let $J, I_1, I_2$ be deals in $P$ such that
$J$ and $I_1$ are grid ideals, $J\subset I_1$, and $I_2 = J:I_1$,
\begin{enumerate}
\item[(a)]  We have ${\GFN}(I_2) = 1$.
\item[(b)]  In particular,  statement (a) holds if
$\mathbb{X}, \mathbb{Y}$ are grid of points, $J= \mathcal{I}(\mathbb{X})$,
$I_1= \mathcal{I}(\mathbb{Y})$, and hence
$I_2= \mathcal{I}(\mathbb{X}\setminus \mathbb{Y})$.
\end{enumerate}
\end{proposition}

\begin{proof}
As Claim (b) is a special case of  (a), let us prove Claim (a).
Since $I_2$ is a grid ideal, we get $\GFN(I_2) = 1$
from Lemma~\ref{uniquebasis}(b),
and the conclusion follows from Theorem~\ref{sameGFan}(d).
\end{proof}

%\begin{proposition}\label{gridminusgrid}
%Let $\mathbb{X}, \mathbb{Y}$ be
%grids of points such that $\mathbb{Y} \subset \mathbb{X}$,
%and let $J$ be the vanishing ideal of $\mathbb{X}{\setminus}\mathbb{Y}$.
%\begin{enumerate}
%\item[(a)]  We have ${\GFN}(J) = 1$.
%\item[(b)]  In particular,  statement (a) holds for a single point $\mathbb{Y}$.
%\end{enumerate}
%\end{proposition}
%
%\begin{proof}
%As  (b) is a special case of  (a), let us prove claim (a).
%Since~$\mathbb{Y}$ is a grid of points, we get
%$\GFN( \mathcal{I}(\mathbb{Y}) )=1$ from Lemma~\ref{uniquebasis}.a,
%and the conclusion follows from Theorem~\ref{sameGFan}.d.
%\end{proof}

%The following example illustrates this proposition.
Let us see an example which illustrates this proposition.

\begin{example}\label{ex-gridminusgrid}
\scriptsize{
\begin{verbatim}
Use P::= QQ[x,y];
F:=x*(x-1)*(x-2)*(x-3)*(x-4);
G:=y*(y-1)*(y-2)*(y-3);
I:=ideal(F,G);
---------------------------------------------------
M:=mat([[0,1], [0,3], [1,1], [1,3], [3,1], [3,3]]);
J1:=IdealOfPoints(P,M);
J2:=Colon(I,J1);
GF:=GroebnerFanIdeals(J2);GF;
-- [ideal(x^2*y^2 -2*x^2*y -6*x*y^2 +12*x*y +8*y^2 -16*y,
--    y^4 -6*y^3 +11*y^2 -6*y, x^5 -10*x^4 +35*x^3 -50*x^2 +24*x)]
Len(GF);
-- 1
-- The ideal J1 is the vanishing ideal of the "white dots".
-- The ideal J2 is the vanishing ideal of the "black dots".
\end{verbatim}
}
\makebox[11 true cm]
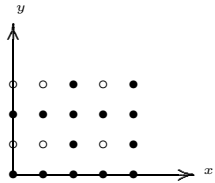
{

\beginpicture
\setcoordinatesystem units <0.4cm,0.4cm>
\setplotarea x from 0 to 5.5, y from 0 to 4.5
\axis left /
\axis bottom /

\arrow <2mm> [.2,.67] from  5.5 0  to 6 0
\arrow <2mm> [.2,.67] from  0 4.5 to 0 5

\put {$\scriptstyle x$} [lt] <0.5mm,0.8mm> at 6.2 0
\put {$\scriptstyle y$} [rb] <1.7mm,0.7mm> at 0 5.2
\put {$\bullet$} at 0 0
\put {$\bullet$} at 1 0
\put {$\bullet$} at 2 0
\put {$\bullet$} at 3 0
\put {$\bullet$} at 4 0
\put {$\circ$} at 0 1
\put {$\circ$} at 1 1
\put {$\bullet$} at 2 1
\put {$\circ$}  at 3 1
\put {$\bullet$} at 4 1
\put {$\bullet$} at 0 2
\put {$\bullet$} at 1 2
\put {$\bullet$} at 2 2
\put {$\bullet$} at 3 2
\put {$\bullet$} at 4 2
\put {$\circ$} at 0 3
\put {$\circ$} at 1 3
\put {$\bullet$} at 2 3
\put {$\circ$} at  3 3
\put {$\bullet$} at 4 3

%\put {$\scriptstyle\times$} at 0 5
%\put {$\scriptstyle\times$} at 1 4
%\put {$\scriptstyle\times$} at 2 3
%\put {$\scriptstyle\times$} at 3 3
%\put {$\scriptstyle\times$} at 4 2
%\put {$\scriptstyle\times$} at 5 1
%\put {$\scriptstyle\times$} at 6 0

\endpicture}
\end{example}

%\medskip
%The following remark indicates a possible direction of future research.
%In particular it gives a hint on how to generalize the theory of  distractions.
%
%
%\begin{remark}\label{equivalentdistraction}
%Definition~\ref{defofdidtrac} can be rephrased by
%substituting the sequences of
%elements in $K$ with sequences of
%linear polynomials of type $x_i-c_{ij}$.
%
%But then we can also try to
%drop the assumption that these polynomials are linear.
%This generalization would allow us to compute distractions
%even in small fields where there are possibly not enough linear polynomials.
%The following examples hints at this direction.
%\end{remark}
%
%\begin{example}\label{ex:char2}
%\scriptsize{
%\begin{verbatim}
%K::=ZZ/(2);
%Use P::= K[x,y];
%A:= ideal(x^5, x^2*y^4, y^7);
%F1:= (x^2+x+1)*(x^3+x+1);
%F2:= (y^3+y^2+1)*(y^4+y+1);
%F3:= (x^2+x+1)*(y^4+y+1);
%I:=ideal(F1,F2,F3);
%LT(I) = A;
%-- true
%GF:=GroebnerFanIdeals(I);
%RGF := [ReducedGBasis(I) | I in GF]; RGF; Len(RGF);
%-- [[x^5 +x^4 +1,  x^2*y^4 +x*y^4 +y^4 +x^2*y +x^2 +x*y +x +y +1,  y^7 +y^6 +y^2 +y +1]]
%-- 1
%\end{verbatim}
%}
%\normalsize
%The ideal $I$ is a generalized distraction of $\langle x^5, x^2y^4, y^7\rangle$ in the
%sense described in Remark~\ref{equivalentdistraction}.
%\end{example}

\bigskip

One of the  main goals of this paper is to identify classes of ideals inside affine $K$-algebras which have a GFan number equal to 1. Using the notions of distractions of ideals and their linear shifts we were able to identify a large class of such ideals and provided a methodology for constructing them. Furthermore, we proved that complementary ideals have the same GFan number which provides a tool for identifying ideals of (large) sets of points as having a GFan number 1 based on the ideal of the (small) complementary set of points. Future work may involve a geometric characterization of all data sets with GFan number equal to~1.

%We end the paper with a problem which is still open. Although distractions and their linear shifts provide a huge amount of ideals of points with GFan number~1, Example~\ref{ex:newExample} shows that they do not cover all ideals of points with this property. So one problem is still open.

\section*{\large Acknowledgements}
We thank Shuhong Gao and Sean Sather-Wagstaff for their insightful comments during fruitful discussions. We also thank Anyu Zhang for the helpful computations she performed. Dimitrova, He, and Stigler were partially supported by the National Science Foundation under Award DMS-1419023. Finally, we are pleased to thank the anonymous referee for many useful comments.

%\medskip
%\textbf{Problem}. Characterize geometrically the sets of points with a unique reduced Gr\"obner basis.

\end{document}